\documentclass[a4paper, 12pt]{article}
\usepackage[a4paper, left=2cm, right=1cm]{geometry}
\usepackage{colortbl}
\usepackage{multirow}
\usepackage{array}
\usepackage{color}
\usepackage{amsmath,amssymb,latexsym}
\usepackage[mathscr]{eucal}
\usepackage{arydshln}
\newcommand{\R}{\mathbb{R}}

\newcommand{\Q}{\mathbb{Q}}
\newcommand{\N}{\mathbb{N}}

\newcommand{\beq}{\begin{equation}}
\newcommand{\eeq}{\end{equation}}
\newcommand{\bea}{\begin{eqnarray}}
\newcommand{\eea}{\end{eqnarray}}
\newcommand{\bean}{\begin{eqnarray*}}
\newcommand{\eean}{\end{eqnarray*}}

\newtheorem{theorem}{Theorem}
\newtheorem{lemma}[theorem]{Lemma}

\newtheorem{pro}[theorem]{Proposition}

\newenvironment{proof}%
{\par\noindent\emph{Proof:\ }}%
{\ \hfill ~\rule{2mm}{2mm}\par\bigskip}
{\par\noindent\textbf{Remark:\ }}%
{\ \hfill \par\bigskip}
\def\mut{{{\mu}}}
\def\muu{{{\nu}}} 

\def\mud{{{\nu}}} 
 
\def\hd{{{h}}}

\def\hpp{\mathfrak{h}}

\def\zzz{ u }
\def\rrr{ v }
\def\ZZ{ {\tilde{u}} }
\def\RR{ {\tilde{v}} } 
\def\serUt{ \widetilde{\mathtt{U}} }
\def\serU{ \mathtt{U} }
\def\facV{ \mathtt{V} }
\def\facVt{ \widetilde{\mathtt{V}} }
\def\IPI{ \mathtt{I} }
\def\IPIt{ \widetilde{\mathtt{I}} }
\def\IPIp{ \hat{\mathtt{I}} }

\def\IPGp{ \hat{\mathtt{G}} }
\def\IPHp{ \hat{\mathtt{H}} }
\def\Rr{ \mathtt{R} }
\def\CurvI{\mathtt{S}}
\def\CurvIt{\widetilde{\mathtt{S}}}
\newcommand\M{{\mathtt{M}}}
\def\Mt{\widetilde{\mathtt{M}}}
\def\Mh{\widehat{\mathtt{M}}}

\def\betat{\widetilde{\beta}}
\def\lo{ l_0 }
\def\mo{ m_0 }
\def\no{ n_0 }
\makeatletter
\newcolumntype{"}{@{\hskip\tabcolsep\vrule width 2pt\hskip\tabcolsep}}
\def\hlinewd#1{%
 \noalign{\ifnum0=`}\fi\hrule \@height #1 \futurelet
 \reserved@a\@xhline}
\makeatother
\def\t{{\mathbf t}}
\def\ttr{{\mathbf t}}
\def\tpl{{\hat{\ttr}}}
\def\nablatr{{\nabla}}
\def\nablapl{{\hat{\nabla}}}
\def\Opl{{\ell}}
 \def\Oplt{{\ell}}
 
 \def\Oplp{\hat{\ell}}
 
 \def\th{\hat{l}}

 \def\qh{\mathscr{P}}
 \def\QH{\mathcal{Q}}

 \def\cofactor{\mathscr{K}}

 \def\FP{{\mathbf P}}
 \def\FPt{\widetilde{\mathbf P}}
 \def\x{{\mathbf x}}

 \def\F{{\mathbf F}}
 \def\G{{\mathbf G}}
 \def\P{{\mathbf P}}
 \def\X{{\mathbf X}}
 \def\zero{{\mathbf 0}}
\def\tx{\widetilde{\x}}

\def\tF{{\widetilde{\mathbf F}}}

\def\FZERO{{\mathbf{D}_0}}
\def\spande{{\mathrm{Span}}}

\def\diverg{{\mathrm{div}}}
\def\Cor{{\mathrm{Cor}}}
\def\Ker{{\mathrm{Ker}}}
\def\Range{{\mathrm{Range}}}

\def\parent#1{\left( #1 \right)}

\def\llave#1{\left\{ #1 \right\}}

\def\vedo#1#2{ \begin{pmatrix} #1 \\ #2 \end{pmatrix} }
\def\vetre#1#2#3{ \begin{pmatrix} #1 \\ #2 \\ #3 \end{pmatrix} }
\def\vedot#1#2{\left( #1 , #2 \right)^{T}}
\def\vetret#1#2#3{ \left( #1 , #2 , #3 \right)^T }
\def\vedor#1#2{ \begin{pmatrix} #1 \\ \hdashline[2pt/2pt] #2 \end{pmatrix} }
\def\vetrer#1#2#3{ \begin{pmatrix} #1 \\ #2 \\ \hdashline[2pt/2pt] #3 \end{pmatrix} }

\def\fracp#1#2{{\textstyle{\frac{#1}{#2}}}}
\date{\today}
\title{ On the Integrability Problem for the Hopf-Zero singularity and its relation with the inverse Jacobi multiplier}
\author{
A. Algaba$^\dagger$, N. Fuentes$^\dagger$, E. Gamero$^\ddagger$, C. Garc\'{\i}a$^\dagger$
\\
$^\dagger$CEAFMC. Faculty of Experimental Sciences\\ University of Huelva,
Spain
\\
$^\ddagger$ Dept. Applied Mathematics II, ETSI.\\ University of
Sevilla, Spain}
\begin{document}
\bibliographystyle{plain}

\maketitle

\abstract{
In this paper we use the orbital normal form of the nondegenerate Hopf-zero singularity to obtain necessary conditions for the existence of first integrals for such singularity.
Also, we analyze the relation between the existence of first integrals and of inverse Jacobi multipliers. 
Some algorithmic procedures for determining the existence of first integrals are presented, and they are applied to some families of vector fields.}

\section{Introduction}
Let us consider an analytic three-dimensional system that undergoes a linear degeneracy corresponding to a zero and a pair of pure imaginary eigenvalues.
By translating the equilibrium point to the origin and using a linear transformation, the Hopf-zero singularity can be written as 
$$
\vetre{\dot x}{\dot y}{\dot z} = \vetre{-y}{x}{0} +\vetre{f(x,y,z)}{g(x,y,z)}{h(x,y,z)},
$$
where $f,g,h$ are analytic functions at the origin that denote the nonlinear terms.

We consider the nondegenerate Hopf-zero singularity, that arises by assuming the generic conditions $\frac{\partial^2 h }{\partial x^2} \neq 0$ or $\frac{\partial^2 h }{\partial y^2} \neq 0$.
Under this hypothesis, it is a simple matter to show that the above system can be expanded in quasi-homogeneous terms of type $\t=(1,1,2)$:
\beq\label{hopf-zero-system}
\vetre{\dot x}{\dot y}{\dot z} = \vetre{-2y}{2x}{x^2+y^2} +\vetre{F_1(x,y,z)}{G_1(x,y,z)}{H_1(x,y,z)} + \vetre{F_2(x,y,z)}{G_2(x,y,z)}{H_2(x,y,z)} +\cdots,
\eeq
where $\F_k=\vetret{F_k(x,y,z)}{G_k(x,y,z)}{H_k(x,y,z)}$ is a quasi-homogeneous vector field of type $\t$ and degree $k$.
We notice that the \emph{principal part} (the lowest-degree quasi-homogeneous term, which has degree 0) can be expressed as 
\beq\label{princpartHZ}
\F_0(\x)=\vetrer{ -2y }{ 2x }{ x^2+y^2 } = \vedor{ \X_{\hpp} }{ \hpp },
\eeq
where $\x=(x,y,z)$ and 
\beq\label{Hamilt}
\X_{\hpp}= \vedot{-\frac{\partial \hpp}{\partial y} }{ \frac{\partial \hpp}{\partial x} }=\vedot{-2y}{2x} 
\eeq
denotes the planar Hamiltonian vector field with Hamiltonian function $\hpp=x^2+y^2$.

In this paper, we study the existence of first integrals in a neighborhood of the equilibrium point at the origin for this kind of systems (recall that a first integral is a non-constant function that is constant when it is evaluated along any solution of the system).

Revealing the existence of first integrals for a given system is very useful to understand its qualitative behavior.
Namely, for a planar system the existence of a first integral determines completely its phase portrait.
For higher-dimensional systems, this can be done by obtaining a sufficient number of functionally-independent first integrals.
 
The three-dimensional center problem for the Hopf-zero singularity (that consists of determining whether there is a neigbourhod of the singularity foliated by periodic orbits, including a curve of equilibria) has been analyzed in \cite{IGa1,IGa2}.
This is equivalent to the integrability problem, that consists of determining the existence of a pair of functionally-independent first integrals. 
In the quoted works, it is shown that a Hopf-zero singularity is completely integrable if, and only if, it is orbitally equivalent to its linear part $\vetret{-y}x0$. 
Moreover, in the case of integrability, there are two functionally-independent first integrals of the form $\IPI_1=\hpp+\cdots$ and $\IPI_2=z+\cdots$ (the dots denote higher-order terms).
 
It is a simple matter to show that, in the nondegenerate Hopf-zero singularity (\ref{hopf-zero-system}) that we are considering, there are no first integrals of the form $\IPI_2=z+\cdots$.
In other words, the nondegenerate Hopf-zero singularity (\ref{hopf-zero-system}) can not be completely integrable and our analysis will be focused on detecting the existence of first integrals for such a singularity of the form $\IPI=\hpp+\cdots$.

This is still a difficult problem and there are few known satisfactory methods to solve it.
In the present paper, we use the orbital normal form for system (\ref{hopf-zero-system}) obtained in \cite{AlgHZNF} to establish necessary conditions for the existence of first integrals for the nondegenerate Hopf-zero singularity (\ref{hopf-zero-system}).
Moreover, we analyze the relation between the existence of first integrals and of inverse Jacobi multipliers. 
In the planar nilpotent case, analogous relations concerning first integrals and inverse integrating factors has been obtained in \cite{ Algaba122,AlGarGa,AGGintnilpfii17,Algaba12}.

This paper is organized as follows.
In Section \ref{Hopzero} we include definitions and results about quasi-homogeneous vector fields and the nondegenerate Hopf-zero orbital normal form, that we use in this work (their proofs can be found in \cite{AlgHZNF}). 
In Section \ref{IntegrabHZ}, the orbital normal form is used to obtain some results about the analytic integrability of this singularity.
The main result is Theorem \ref{teHzinteg}, that determines the existence of an analytical first integral for the nondegenerate Hopf-zero singularity in terms of the vanishing of some normal form coefficients. 
Moreover, an algorithmic procedure for obtaining necessary conditions for the existence of first integrals of polynomial vector fields, that is applicable under some hypothesis in the orbital normal form, is presented.
Section \ref{ExistenceJacobi} analyzes the relation between the existence of first integrals and the existence of inverse Jacobi multipliers for the nondegenerate Hopf-zero singularity. 
In particular, based on an algorithm to determine the existence of inverse Jacobi multipliers, we present a new algorithmic procedure to determine the existence of first integrals of polynomial vector fields, that is applicable in all the cases.
Finally, in Section \ref{ParticCases} we apply the results to a couple of three-parameter families of vector fields, where we find all the cases of existence of analytical first integrals.

\section{The Hopf-zero orbital normal form}
\label{Hopzero}

In this section, we collect results from \cite{AlgHZNF} that will be used along this paper. 
Among them, the main result is Theorem \ref{NFHopf-Zero-}, where we present the orbital normal form for system (\ref{hopf-zero-system}).
Before we state it, we introduce some definitions and results.

We say that a scalar function $f$ of $n$ variables is quasi-homogeneous of type $\t=(t_1, \dots, t_n)\in\N^n$ and degree $k$ if $f(\epsilon^{t_1} x_1, \dots , \epsilon^{t_n} x_n )= \epsilon^k f(x_1, \dots, x_n)$.
A vector field $\F=(F_1, \dots, F_n)^T$ is quasi-homogeneous of type $\t$ and degree $k$ if $F_j\in \qh^{\t}_{k+t_j}$ for $j=1, \dots , n$.
The vector spaces of quasi-homogeneous functions and vector fields of type $\t$ and degree $k$ are denoted, respectively, by $\qh^{\t}_k$ and $\QH_k^{\t}$.

In this paper, we use the type $\ttr=(1,1,2)$ (for functions and vector fields depending on three variables), as well as the type $\tpl=(1,1)$ (, that appear when dealing with functions and vector fields depending on two variables.
For instance, we have $\F_0\in\QH_0^{\ttr}$, $\X_\hpp\in\QH_0^{\tpl}$ and $\hpp\in\qh_{2}^{\tpl}$.

A conservative-dissipative decomposition of quasi-homogeneous planar vector fields has been used in \cite{AlgabaNonlinearity09} in the study of the integrability problem.
In the specific case that we are considering, this decomposition reads as follows.

\begin{pro}\label{con-dis}
Let us consider $\P_k\in\QH_{k}^{\tpl}$ and denote $\FZERO=\vedot{x}{y} \in \QH_{0}^{\tpl}$.
Then, there exist unique quasi-homogeneous polynomials $\hd_{k+2}\in\qh_{k+2}^{\tpl}$ and $\mud_k \in \qh_{k}^{\tpl}$ such that:
\beq\label{condis}
\P_k=\X_{\hd_{k+2}} + \mud_k \, \FZERO.
\eeq
Moreover, $\hd_{k+2} = \frac{1}{k+2}\FZERO\wedge\P_{k}$ and $\mud_k = \frac{1}{k+2}\diverg(\P_{k})$.
\end{pro}

In the above proposition, we have introduced the wedge product of two planar vector fields
$\F=(P,Q)^{T} $, $\G=(R,S)^{T} $, defined by $ \F \wedge \G = P \, S - Q \, R $ (see \cite{Guckenheimer83b}) and the divergence $\diverg\parent{ \F } = \frac{\partial P}{\partial x}+\frac{\partial Q}{\partial y}$.
We notice that, if $\F\in \QH_{k}^\tpl$ and $\G \in \QH_{l}^\tpl$, then $ \F \wedge \G \in \qh_{k+l+2}^{\tpl}$ and 
$\diverg\parent{ \F } \in\qh_k^\tpl$.

Next, we present the orbital normal form for system (\ref{hopf-zero-system}) obtained in \cite{AlgHZNF}.
It determines how much system (\ref{hopf-zero-system}) can be simplified by means of an infinite sequence of time-reparametrizations and near-identity coordinate transformations.
In fact, the orbital normal form presented is formal, which indicates that we will not discuss matters of convergence.

\begin{theorem}\label{NFHopf-Zero-}
A formal normal form under orbital equivalence for system (\ref{hopf-zero-system}) is 
\beq\label{FNHZEquiv}
\dot{\x} =\G(\x)=\F_0(\x) + \vedor{ G_1(z)\FZERO }{ G_2(z) } ,
\eeq
where $G_1(z)= {\sum_{k \geq 1}^{\infty}} a_{k}z^{k}$ and $G_2(z)= {\sum_{k \geq 1}^{\infty}} b_{k}z^{k+1}$.
\end{theorem}

Observe that the $0$th-degree quasi-homogeneous term of the orbital normal form (\ref{FNHZEquiv}) agrees with those of system (\ref{hopf-zero-system}).
Moreover, the $k$th-degree quasi-homogeneous term of the orbital normal form (\ref{FNHZEquiv}) is
$$ 
\G_k(\x)=\vedor{ a_{k}z^{k}\FZERO }{ b_{k}z^{k+1} } \in\QH_k^\ttr.
$$ 

In the rest of this paper we use the orbital normal form (\ref{FNHZEquiv}) to study the integrability problem for the nondegenerate Hopf-zero singularity.
The first nonzero term in the Taylor expansions of functions $G_1$ and $G_2$ play an outstanding role in this analysis.
Therefore, let us denote
\bea
\nonumber
\lo&:=&\min\llave{l\in\N: a_l\neq 0},\\
\label{l0m0}
\mo&:=&\min\llave{m\in\N: b_m\neq 0}.
\eea
We notice that $G_1(\zzz )\equiv 0$ if, and only if, $\lo=+\infty$, and $G_2(\zzz )\equiv 0$ if, and only if, $\mo=+\infty$.

\section{The integrability problem for the Hopf-zero singularity}
\label{IntegrabHZ}

In this subsection, we show that the orbital normal form (\ref{FNHZEquiv}) is useful in the analysis of the integrability problem (consisting of determining the existence of a first integral) for the nondegenerate Hopf-zero singularity (\ref{hopf-zero-system}).

Recall that a function $\IPI$ is called a first integral for system (\ref{hopf-zero-system}) if $\IPI$ is constant when it is evaluated along any solution of the system.
If $\IPI$ is a ${\cal{C}}^1$ function, using the chain rule, this means that $\nablatr \IPI \cdot \F = 0$.

Our first result states that the analysis of the integrability problem for analytic systems can be reduced to the formal context through formal diffeomorphisms.

\begin{pro}\label{intanaformal} 
Let us consider the system $\dot \x=\F(\x)$, where $\F$ is an analytic vector field and the transformation $\x=\Phi(\tx)$, where $\Phi$ is a formal diffeomorphism. 
Then, system $\dot \x=\F(\x)$ admits an analytical first integral if, and only if, the transformed system $\dot \tx=\tF(\tx)$  admits a formal first integral.
\end{pro}

\begin{proof} 
The necessary condition is trivial, because if $\IPI$ is an analytical first integral for system $\dot\x=\F(\x)$, then $\IPIt:=\IPI\circ\Phi$ is a formal first integral for system $\dot\tx={\tF}(\tx)$.

To prove the sufficient condition, let us denote by $\IPIt$ a formal first integral of system $\dot\tx={\tF}(\tx)$.
Then, $\IPIp=\IPIt\circ\Phi^{-1}$ is a formal first integral of system $\dot\x=\F(\x)$.
From Theorem A of \cite{MatteiMoussu80}, there exists a formal scalar function $\th$ such that $\th(0)=0$, ${\th}\, '(0)=1$, such that $\IPI=\th\circ\IPIp$ is an analytical first integral for system $\dot\x=\F(\x)$.
\end{proof}

We notice that the orbital normal form (\ref{FNHZEquiv}), as well as those obtained in \cite{CWY2003,CWY2005}, is invariant under rotations. 
Hence, the first integrals depend on $x^2+y^2$ and $z$.
Next result uses the orbital normal form (\ref{FNHZEquiv}) to determine the existence of a formal first integral for the nondegenerate Hopf-zero singularity by reducing it to a nilpotent singularity.
If we use instead the orbital normal form given in \cite{CWY2003,CWY2005}, then an integrability problem for a planar system with null linear part arises, which is more difficult to solve.

\begin{pro}\label{proHZTB} 
The orbital normal form (\ref{FNHZEquiv}) admits a formal first integral if, and only if, the planar system
\bea\label{FNparecidoTakens}
\nonumber{\dot{\zzz}}&=&{\rrr+G_2(\zzz)},\\
{\dot{\rrr}}&=&{2 \rrr G_1(\zzz)},
\eea
is formally integrable.
\end{pro}

\begin{proof} 
Let us consider cylindrical coordinates $x=\rho\sin(\theta)$, $y=\rho\cos(\theta)$, $z=\zzz $, and the singular change $\rrr =\rho^2$.
Then, the normal form (\ref{FNHZEquiv}) becomes:
\bean 
\nonumber
{\dot{\zzz }}&=&{\rrr +G_2(\zzz )},\\
{\dot{\rrr }}&=&{2 \rrr G_1(\zzz )},\\
\nonumber
\dot{\theta}&=&2.
\eean
It is enough to remove the azimuthal component to complete the proof.
\end{proof}

We denote the vector field corresponding to planar system (\ref{FNparecidoTakens}) by 
\beq\label{Pe}
\FP(u,v)={\vedo{\rrr +G_2(\zzz )}{2 \rrr G_1(\zzz )}}.
\eeq

Next result provides a necessary condition for the existence of analytical first integrals for system (\ref{hopf-zero-system}), which determines the structure of the quasi-homogeneous normal form in case of existence of analytical first integrals.

\begin{pro}\label{proCNintHZ} 
Let us assume that system (\ref{hopf-zero-system}) admits an analytical first integral. 
Then, its formal orbital normal form (\ref{FNHZEquiv}) is given in one of the following items:
\begin{description}
\item[(a)] 
$\dot{\x}=\F_0(\x)$.
\item[(b)] 
$\dot{\x}=\F_0(\x)+\F_s(\x)+\cdots$, where $s\in\N$ and $\F_s\in\QH_s^{\ttr}$ is one of the following vector fields:
\begin{description}
\item[(b.1)] 
$\F_s(\x) = \vedor{ a_{\lo} z ^{\lo}\FZERO }{ 0 }\in\QH_s^\ttr$, where $s=2\lo$ and $a_{\lo}\in\R\setminus\llave{0}$.
\item[(b.2)] 
$\F_s(\x) = \vedor{ \zero }{ b_{\mo} z ^{\mo+1} }\in\QH_s^\ttr$, where $s=2\mo$ and $b_{\mo}\in\R\setminus\llave{0}$.
\item[(b.3)] 
$\F_s(\x) = \vedor{ a_{\mo} z ^{\mo} \FZERO }{ b_{\mo} z ^{\mo+1} }\in\QH_s^\ttr$, where $s=2\mo$ and $a_{\mo},b_{\mo}\in\R\setminus\llave{0}$ satisfy
$$
2n_1a_{\mo}+(\mo+1)n_2b_{\mo}=0,
$$ 
for some $n_1,n_2\in\N$ coprime (i.e., their greatest common divisor is 1).
\end{description}
\end{description}
\end{pro}

\begin{proof} 
From Proposition \ref{intanaformal}, we obtain that system (\ref{hopf-zero-system}) admits an analytic first integral if, and only if, its orbital normal form (\ref{FNHZEquiv}) has a formal first integral. 
From Proposition \ref{proHZTB}, this occurs if, and only if, the planar system (\ref{FNparecidoTakens}) is formally integrable. 
In this case, if we select an arbitrary type $\tpl\in\N^2$, the principal part of the vector field $\FP$ given in (\ref{Pe}) must be polynomially integrable.
Let us consider the following situations:
\begin{itemize}
\item 
Case $G_1(\zzz )\equiv G_2(\zzz )\equiv 0$. 
Let us take the type $\tpl=(1,1)$. 
Then, the principal part of $\FP$ is $\FP_0(\zzz ,\rrr )=(\rrr ,0)$ and $\IPIp=\rrr $ is an analytic first integral. 
This case is considered in item {\bf (a)} of the statement.

\item 
Case $G_1(\zzz )\not\equiv 0$ or $G_2(\zzz )\not \equiv 0$. 
Let us denote
\bea
\label{l0m0n0}
\no&:=&\min\llave{n\in\N: 2a_n+(n+1)b_n\neq 0}.
\eea
We notice that $\min\llave{\mo,\lo}\leq \no$ and it is possible that $\lo=+\infty$, $\mo=+\infty$ or $\no=+\infty$, but the situation $\lo=\mo=+\infty$ can not occur.
The following sub-cases can arise:
\begin{itemize}
\item 
If $\lo<\mo$, taking the type $\tpl=(1,\lo+1)$, the principal part of $\FP$ is 
$$ 
\FP_{\lo}(\zzz ,\rrr )=(\rrr ,2a_{\lo}\zzz ^{\lo}\rrr )^T \in \QH_{\lo}^{\tpl},
$$
and $\IPIp=\rrr -\fracp{2a_{\lo}}{\lo+1}\zzz ^{\lo+1}$ is an analytic first integral of $\FP_{\lo}$.
This case corresponds to item {\bf (b.1)}.
\item 
If $\mo<\lo$, taking the type $\tpl=(1,\mo+1)$, the principal part of $\FP$ is 
$$
\FP_{\mo}(\zzz ,\rrr )=(\rrr +b_{\mo}\zzz ^{\mo+1},0)^T\in\QH_{\mo}^{\tpl},
$$
and $\IPIp=\rrr $ is an analytic first integral of $\FP_{\mo}$.
This is the case {\bf (b.2)}.
\item 
If $\mo = \lo < \no$, then we have $2a_{\mo}+(\mo+1)b_{\mo}=0$. 
Taking the type $\tpl=(1,\mo+1)$, the principal part of $\FP$ is 
$$
\FP_{\mo}(\zzz ,\rrr )=(\rrr +b_{\mo}\zzz ^{\mo+1},-(\mo+1)b_{\mo} \zzz ^{\mo}\rrr )^T \in\QH_{\mo}^{\tpl},
$$
which is a Hamiltonian vector field, with Hamiltonian function 
${-b_{\mo}\zzz ^{\mo+1}\rrr -\fracp{1}{2}\rrr ^2}$. 
Therefore, $\FP_{\mo}$ is polynomially integrable.
This case is presented in item {\bf (b.3)} with $n_1=n_2=1$.
\item 
If $\mo = \lo = \no<+\infty$, then we have $2a_{\mo}+(\mo+1)b_{\mo}\neq 0$.
Taking the type $\tpl=(1,\mo+1)$, the principal part of $\FP$ is 
 $$
 \FP_{\mo}(\zzz ,\rrr )=(\rrr +b_{\mo}\zzz ^{\mo+1},2a_{\mo} \zzz ^{\mo}\rrr )^T\in\QH_{\mo}^{\tpl},
 $$
which is integrable because $\FP$ also is.

Let us consider the conservative-dissipative splitting (\ref{condis}) for $\FP_{\mo}$. 
The Hamiltonian function in the quoted splitting is
$$
{\hd}=-\fracp{1}{2}\rrr \parent{\rrr -\fracp{2a_{\mo}-(\mo+1)b_{\mo}}{2(\mo+1)}\zzz ^{\mo+1}}.
$$
Then, 
$$
\CurvI_1\equiv \rrr =0, \textrm{ and } \CurvI_2\equiv \rrr -\fracp{2a_{\mo}-(\mo+1)b_{\mo}}{2(\mo+1)}\zzz ^{\mo+1}=0,
$$
are invariant curves of $\FP_{\mo}$, with cofactors $\cofactor_1=2a_{\mo}\zzz ^{\mo}$ and $\cofactor_2=(\mo+1)b_{\mo}\zzz ^{\mo}$, respectively.
Since $\FP_{\mo}$ is polynomially integrable, there exist $n_1,n_2\in\N$ coprime such that $\CurvI_1^{n_1}\CurvI_2^{n_2}$ is a polynomial first integral of $\FP_{\mo}$, i.e.
$$ 
n_1 \cofactor_1 + n_2 \cofactor_2 = \parent{2 n_1 a_{\mo} + n_2(\mo+1) b_{\mo}}\zzz ^{\mo}=0.
$$
This case is presented also in item {\bf (b.3)}.
\end{itemize}
\end{itemize}
\end{proof}

Next, we present a necessary and sufficient condition for the existence of an analytic first integral of system (\ref{hopf-zero-system}).

\begin{theorem}\label{teHzinteg}
System (\ref{hopf-zero-system}) admits an analytical first integral if, and only if, its formal orbital normal form (\ref{FNHZEquiv}) is given in one of the following items:
\begin{description}
\item[(a)] 
$\dot{\x}=\F_0(\x)$. 
In this case, there exists a first integral of the form $\IPI=\hpp+\cdots$. 
Moreover, there is a curve of equilibria passing through the origin surrounded by an infinity of invariant cylinders.
\item[(b.1)]
$\dot{\x} = \F_0(\x) + \vedor{ G_1(z )\FZERO }{ 0 }$, with $G_1(z )=\sum_{k \geq 1} a_k z ^k$.
In this case, there exists a first integral of the form $\IPI=\hpp+\cdots$.
Moreover, there is a curve of equilibria passing through the origin.
\item[(b.2)]
$\dot{\x} = \F_0(\x) + \vedor{ \zero }{ G_2(z ) } $, with $G_2(z )=\sum_{k \geq 1} b_k z ^{k+1}$.
In this case, there exists a first integral of the form $\IPI=\hpp + \cdots$.
\item[(b.3)]
$\dot{\x} = \F_0(\x) + \F_{\mo}(\x)$, where $\F_{\mo}(\x) = \vedor{ a_{\mo} z ^{\mo} \FZERO }{ b_{\mo} z ^{\mo+1} } $, $\mo\in\N$, and $a_{\mo},b_{\mo}\in\R\setminus\llave{0}$ satisfy $2n_1a_{\mo}+(\mo+1)n_2b_{\mo}=0$ for some coprime natural numbers $n_1,n_2$. 
In this case, there exists a first integral of the form $\IPI=\hpp^{n_1+n_2}+\cdots$.
\end{description}
In the above expressions, the dots denote higher-order quasi-homogeneous terms.
\end{theorem}

\begin{proof} 
Firstly, we prove the sufficient condition. 
\begin{description}
\item[(a)] 
The normal form $\dot{\x} = \F_0(\x) $ has the first integral $\IPIt=\hpp$.
Undoing the normalizing transformations, we obtain that $\IPI=\hpp+\cdots$ is a first integral for system (\ref{hopf-zero-system}).
\item[(b.1)] 
The normal form $\dot{\x} = \F_0(\x) + \vedor{ G_1(z )\FZERO }{ 0 }$ admits the first integral $\IPIt=\hpp - 2 \int_0^{z } G_1(\xi)d\xi$.
Undoing the normalizing transformations, we obtain that $\IPI = \hpp + \cdots$ is a first integral for system (\ref{hopf-zero-system}).
\item[(b.2)] 
The normal form $\dot{\x} = \F_0(\x) + \vedor{ \zero }{ G_2(z ) }$ has the first integral $\IPIt=\hpp$.
Undoing the normalizing transformations, we obtain that $\IPI = \hpp + \cdots$ is a first integral for system (\ref{hopf-zero-system}).
\item[(b.3)] 
If $n_1=n_2=1$ (Hamiltonian case), then $\IPIt=\hpp^2+2b_{\mo} z ^{\mo+1} \hpp$ is a first integral for the normal form $\dot{\x} = \F_0(\x) + \F_{\mo}(\x)$. 
Undoing the normalizing transformations, we obtain that $\IPI=\hpp^2+\cdots$ is a first integral for system (\ref{hopf-zero-system}).

Otherwise (dissipative case), the normal form $\dot{\x} = \F_0(\x) + \F_{\mo}(\x)$ admits the first integral $\IPIt=\hpp^{n_1}\parent{\hpp-\fracp{2a_{\mo}-(\mo+1)b_{\mo}}{2(\mo+1)}z ^{\mo+1}}^{n_2}$. 
Undoing the normalizing transformations, we obtain that $\IPI=\hpp^{n_1+n_2}+\cdots$ is a first integral for system (\ref{hopf-zero-system}).
\end{description}
Next, we prove the necessary condition.
Let us assume that system (\ref{hopf-zero-system}) admits an analytical first integral and consider its orbital normal form (\ref{FNHZEquiv}).

From Proposition \ref{proCNintHZ}, the quoted formal orbital normal form is either $\dot{\x}=\F_0(\x)$ (that corresponds to the item {\bf (a)} of the statement) or $\dot{\x}=\F_0(\x)+\F_s(\x)+\cdots$, where $\F_s\in\QH_s^\ttr$ is given in one of the cases of item {\bf (b)} of Proposition \ref{proCNintHZ}.
We deal with each case separately.
\begin{description}
\item[(b.1)] 
Here, $\F_s(\x) = \vedor{ a_{\lo} z ^{\lo}\FZERO }{ 0 }\in\QH_s^\ttr$, where $s=2\lo$, and $a_{\lo}\in\R\setminus\llave{0}$. 
To complete the proof in this case, it is enough to show that $G_2(z )\equiv 0$ in the orbital normal form (\ref{FNHZEquiv}). 

We use \emph{reductio ad absurdum}: if $G_2(z )\not\equiv 0$ then $\mo<+\infty$.
Taking the type $\tpl=(1,\lo+1)$, the principal part of planar system (\ref{FNparecidoTakens}) is
$$ 
\FP_{\lo}(\zzz ,\rrr )=(\rrr ,2a_{\lo}\zzz ^{\lo}\rrr )^T \in \QH_{\lo}^{\tpl}.
$$
Let us denote the formal first integral of system (\ref{FNparecidoTakens}) by $\IPIp$.
We notice that $\IPGp(\zzz ,\rrr )=\rrr -2\int_0^{\zzz } G_1(\xi)d\xi$ is a first integral of the vector field $(\rrr ,2 G_1(\zzz )\rrr )^T$.
If we define $\IPHp=\IPIp-\IPGp$, we have:
$$
\nablatr \IPIp\cdot\FP=\nablatr (\IPGp+\IPHp)\cdot\FP=\nablatr \IPGp\cdot\FP+\nablatr \IPHp\cdot\FP
=\nablatr \IPGp\cdot (G_2(\zzz ),0)^T+\nablatr \IPHp\cdot\FP=0.
$$
In this equality, the quasi-homogeneous term of degree $\lo+\mo+1$ is given by
$$ 
-2a_{\lo}b_{\mo}\zzz ^{\lo+\mo+1}+\nablatr \IPHp_{\mo+1}\cdot\FP_{\lo}=-2a_{\lo}b_{\mo}\zzz ^{\lo+\mo+1}+\rrr \nablatr \IPHp_{\mo+1}\cdot(1,2a_{\lo} \zzz ^\lo)^T=0,
$$
where $\IPHp_{\mo+1}\in\qh_{\mo+1}^\tpl$ is the quasi-homogeneous term of degree $(\mo+1)$ of $\IPHp$.
Nevertheless, the above equation is incompatible. 
Hence, system (\ref{FNparecidoTakens}) can not admit any first integral and, applying Proposition \ref{proHZTB}, system (\ref{hopf-zero-system}) does not admit any formal first integral, which is contradictory.

\item[(b.2)] 
Now, $\F_s(\x) = \vedor{ \zero }{ b_{\mo} z ^{\mo+1} }\in\QH_s^\ttr$, where $s=2\mo$, 
and $b_{\mo}\in\R\setminus\llave{0}$. 
To complete the proof in this case, it is enough to show that $G_1(z )\equiv 0$ in the orbital normal form (\ref{FNHZEquiv}). 

Again, we use \emph{reductio ad absurdum}:  if $G_1(z )\not\equiv 0$ then $\lo <+\infty$.
Taking the type $\tpl=(1,\mo+1)$, the principal part of planar system (\ref{FNparecidoTakens}) is
 $$
 \FP_{\mo}(\zzz ,\rrr )=(\rrr +b_{\mo}\zzz ^{\mo+1},0)^T\in\QH_{\mo}^{\tpl}.
 $$
Let us denote the formal first integral of system (\ref{FNparecidoTakens}) by $\IPIp$.
We notice that $\IPGp(\zzz ,\rrr )=\rrr $ is a first integral of the vector field $(\rrr +G_2(\zzz ),0)^T$.
If we define $\IPHp=\IPIp-\IPGp$, then:
$$ 
\nablatr \IPIp\cdot\FP=\nablatr(\IPGp+\IPHp)\cdot\FP=\nablatr \IPGp\cdot\FP+\nablatr \IPHp\cdot\FP=\nablatr \IPGp\cdot (0,2 G_1(\zzz )\rrr )^T+\nablatr \IPHp\cdot\FP=0.
$$
In this equality, the quasi-homogeneous term of degree $\lo+\mo+1$ is given by
$$
2a_{\lo}\zzz ^{\lo}\rrr +\nablatr \IPHp_{\mo+1}\cdot\FP_{\mo}=2a_{\lo}\zzz ^{\lo}\rrr +\fracp{\partial \IPHp_{\mo+1}}{\partial \zzz }(\rrr +b_{\mo}\zzz ^{\mo+1})=0,
$$
where $\IPHp_{\mo+1}\in\qh_{\mo+1}^\tpl$ is the quasi-homogeneous term of degree $(\mo+1)$ of $\IPHp$.
As in the previous subcase, the above equation is incompatible, 
and system (\ref{FNparecidoTakens}) can not admit any first integral. 
Hence, by applying Proposition \ref{proHZTB}, we deduce that system (\ref{hopf-zero-system}) 
does not admit any formal first integral, which is contradictory.

\item[(b.3)] 
In this case, we have $\F_{\mo}(\x) = \vedor{ a_{\mo} z ^{\mo} \FZERO }{ b_{\mo} z ^{\mo+1} }\in\QH_s^\ttr $, where $\mo\in\N$, and $a_{\mo},b_{\mo}\in\R\setminus\llave{0}$ satisfy $2n_1a_{\mo}+(\mo+1)n_2b_{\mo}=0$ being $n_1,n_2\in\N$ coprimes.

There are two cases to be considered. 
The first one corresponds to the free-divergence case, that arises if $n_1=n_2=1$, and then $2a_{\mo}+(\mo+1)b_{\mo}=0$.
Taking the type $\tpl=(1,\mo+1)$, the principal part of planar system (\ref{FNparecidoTakens}) is
$$
\FP_{\mo}(\zzz ,\rrr )=(\rrr +b_{\mo}\zzz ^{\mo+1},-(\mo+1)b_{\mo} \zzz ^{\mo}\rrr )^T \in\QH_{\mo}^{\tpl}.
$$
This is a Hamiltonian vector field, with Hamiltonian function ${-b_{\mo}\zzz ^{\mo+1}\rrr -\fracp{1}{2}\rrr ^2}$.
Using Corollary 4.23 of \cite{AlgabaNonlinearity09}, we obtain that system (\ref{FNparecidoTakens}) is integrable if, and only if, it is formally equivalent to $(\dot \zzz ,\dot \rrr ) =\FP_{\mo}(\zzz ,\rrr )$.
Then, from Proposition \ref{proHZTB} we obtain that system (\ref{hopf-zero-system}) admits a first integral if, and only if, it is formally equivalent to $\dot{\x} = \F_0 (\x) +\F_{\mo}(\x).$
This falls into the item \textbf{(b.3)} of the statement.

In the second case (non-zero divergence), we have $2n_1a_{\mo}+(\mo+1)n_2b_{\mo}=0$ for some $n_1,n_2\in\N$ coprimes with $n_1\neq 1$ or $n_2\neq 1$.
Taking the type $\tpl=(1,\mo+1)$, the principal part of planar system (\ref{FNparecidoTakens}) is
$$
\FP_{\mo}(\zzz ,\rrr )=(\rrr +b_{\mo}\zzz ^{\mo+1},2a_{\mo} \zzz ^{\mo}\rrr )^T\in\QH_{\mo}^{\tpl},
$$
which is integrable because $\rrr ^{n_1}\parent{\rrr -\fracp{2a_{\mo}-(\mo+1)b_{\mo}}{2(\mo+1)}\zzz ^{\mo+1}}^{n_2}$ is a first integral and besides $\diverg\parent{\F_{\mo}}=2a_{\mo}+(\mo+1)b_{\mo}\neq 0$.

Using Theorem 1.2 of \cite{AGGintnilp16}, we obtain that $\FP$ is orbitally equivalent to $\FP_{\mo}$.
This case corresponds to the item \textbf{(b.3)} of the statement.
\end{description}
\end{proof}

Using the above theorem, an algorithm providing necessary conditions for the integrability of a polynomial vector field can be derived in cases \textbf{(a)}, \textbf{(b.1)} and \textbf{(b.2)} of Theorem \ref{teHzinteg}. 
First, we present a technical lemma.

\begin{lemma}\label{algoritmoIPI} 
Let us consider the vector field $\F$ associated with system (\ref{hopf-zero-system}). 
Then, there exists a unique scalar function $\IPI=\hpp+\sum_{k\geq 3}\IPI_k$, where $\IPI_k\in\qh_k^\t$, such that the term $\hpp^k$ is missing in $\IPI_{2k}$, for all $k$, that verifies
\beq\label{ecualgI}
\nablatr \IPI\,\cdot\F =\sum_{k\geq 2} \alpha_k z^k,
\eeq
where $\alpha_k\in\R$. 
\end{lemma}

\begin{proof}
The $k$-degree quasi-homogeneous term of the left-hand side of equality (\ref{ecualgI}) is 
$$
 \parent{\nablatr \IPI\,\cdot\F }_k=\nablatr \IPI_k\cdot\F_0 +   \Rr_k=\Oplt_{k}\parent{\IPI_k} +   \Rr_k,
 $$
where
$$ 
  \Rr_k=\sum_{j=2}^{k-1}\parent{\nablatr \IPI_j\,\cdot\F_{k-j} }\in\qh_k^\t,
$$
and we have introduced the Lie derivative operator associated with the principal part of the 
vector field (\ref{hopf-zero-system}), which is defined by:
\bea
\Opl_{k} &:& \qh_{k-r}^\t \longrightarrow \qh_{k}^\t \label{Opele} \\
&& \, \muu_{k-r}\longrightarrow\Opl_{k}\parent{ \muu_{k-r}}=\nabla \muu_{k-r}\cdot\F_{0}.
\nonumber
\eea

Reasoning as in the classical Normal Form Theory, it is possible to choose $\IPI_k$ in order to annihilate the part of $\IPI_k$ 
belonging to the range of the operator $\Oplt_{k}$. 
In such a way, we can achieve 
$$
\parent{\nablatr \IPI\,\cdot\F }_k =   \Rr_k^c\in\Cor\parent{\Oplt_{k}},
$$
a complementary subspace to $\Range(\Oplt_{k})$.
In \cite{AlgHZNF}, it is obtained that, if $k$ is even, then $\Cor(\Oplt_{k}) =\spande\llave{z^{k/2}}$
and $\Ker\parent{\Oplt_{k}}=\spande\llave{ \hpp^{k/2} }$, whereas if $k$ is odd, then $\Cor(\Oplt_{2k_1+1})=\Ker\parent{\Oplt_{k}}=\llave{0}$.

To complete the proof, it is enough to use that, if $k$ is odd, then $  \Rr_k^c=0$, whereas if $k$ is even, we get $  \Rr_k^c=\alpha_i z^i$ where $i= k/2$.
Moreover, in this last case the term $(x^2+y^2)^{k/2}$ can be dropped from the expression of $\IPI_k$ because 
$\Ker\parent{\Oplt_{k}}=\spande\llave{ \hpp^{k/2} }$.
\end{proof}

\begin{theorem}\label{ggg} 
Let us consider system (\ref{hopf-zero-system}) and assume that its formal orbital normal form (\ref{FNHZEquiv}) falls into the cases \textbf{(a)}, \textbf{(b.1)} or \textbf{(b.2)} of Proposition \ref{proCNintHZ}.
Let also consider the unique scalar function $\IPI$ introduced in Lemma \ref{algoritmoIPI} satisfying (\ref{ecualgI}).
Then, system (\ref{hopf-zero-system}) admits an analytical first integral if, and only if, $\alpha_k=0$ for all $k$.
\end{theorem}

\begin{proof}
The necessary condition is trivial:
If $\alpha_k=0$ for all $k$, then from (\ref{ecualgI}) we find that $\nablatr \IPI\,\cdot\F =0$, i.e., 
$\IPI$ is a formal first integral.
From Proposition \ref{intanaformal}, we deduce that system (\ref{hopf-zero-system}) admits an analytical first integral.

Let us prove the sufficient condition.
If system (\ref{hopf-zero-system}) admits an analytical first integral, from Theorem \ref{teHzinteg} we obtain that it also admits a first integral of the form $\IPI=\hpp + \cdots$.
Then, $\IPIt=\IPI -\sum_{k\geq 2} \beta_k \IPI^k$ is a formal first integral for all $\beta_k$, i.e., $\nablatr \IPIt\,\cdot\F =0$.
It is enough to select $\beta_k$ such that the term $\hpp^k$ is missing in the quasi-homogeneous terms of degree $2k$ in $\IPIt$ to obtain the result.
\end{proof}

This theorem allows to define an algorithm for obtaining necessary conditions for the integrability of a polynomial vector field.
Namely, it is enough to look for the unique function of the form $\IPI=\hpp + \cdots$ (specified in Lemma \ref{algoritmoIPI}) and then discard cases of non-integrability from the conditions $\alpha_2\neq 0$, ...

Nevertheless, this algorithm is not applicable for vector fields whose orbital normal form (\ref{FNHZEquiv}) falls into the case \textbf{(b.3)} because in this case the first integral $\IPI=\hpp^{n_1+n_2}+\cdots$ is not determined since the values $n_1,n_2$ are unknown.
 
In the next subsection, we present a different approach, based on the existence of an inverse Jacobi multiplier, 
that overcomes this difficulty.

\section{Relation between the integrability and the existence of an inverse Jacobi multiplier}
\label{ExistenceJacobi}

In this subsection, we study the relation between the integrability of the nondegenerate Hopf-zero singularity and the existence of an inverse Jacobi multiplier for such a singularity (see \cite{BG03,BGM1,BGM2}).

In particular, an algorithm to determine three-dimensional integrable vector fields, based on the use of scalar functions, can be derived.
We recall that an inverse Jacobi multiplier for system (\ref{hopf-zero-system}) is a smooth function $\M$ which satisfies $\nablatr \M \cdot \F = \diverg(\F)\, \M$ in some neighborhood of the equilibrium at the origin.
We observe that, if $\M$ does not vanish in a open set, then the above equality is equivalent to $ \diverg\parent{\frac 1 \M \F}=0$.
For planar systems, the inverse Jacobi multipliers are usually referred as inverse integrating factors.

\begin{lemma}\label{leNofiir} 
System (\ref{FNparecidoTakens}), with $G_1(\zzz )G_2(\zzz )\not\equiv 0$, does not admit any inverse integrating factor of the form $\facV=\rrr +\cdots$.
\end{lemma}

\begin{proof} 
We use \emph{reductio ad absurdum}.
Let us suppose on the contrary that there exist an inverse integrating factor of system (\ref{FNparecidoTakens}) of the form $\facV=\rrr +\cdots$.
Then, $\facV=0$ is an invariant curve of the quoted system.

Let us consider $\lo, \mo, \no$ defined in (\ref{l0m0}), (\ref{l0m0n0}).
We recall that $\min\llave{\mo,\lo}\leq \no$.
We consider the following sub-cases:

\begin{itemize} 
\item $\lo<\mo$. 
Taking the type $\tpl=(1,\lo+1)$ the principal part of planar system (\ref{FNparecidoTakens}) is 
$$
\FP_{\lo}(\zzz ,\rrr )=(\rrr ,2a_{\lo}\zzz ^{\lo}\rrr )^T \in \QH_{\lo}^{\tpl}.
$$

Let us denote $\CurvI_1=\rrr$. 
It is trivial to show that $\CurvI_1=0$ is an invariant curve of system (\ref{FNparecidoTakens}).
Moreover, in \cite{AGGintnilp16,AGGintnilpfii17,Algaba18} it is shown that system (\ref{FNparecidoTakens}) has an invariant curve of the form $\CurvI_2=0$, where  $\CurvI_2=\rrr -\fracp{2a_{\lo}}{\lo+1}\zzz ^{\lo+1}+\cdots$, and $\CurvI_1=0$, $\CurvI_2=0$ are the unique irreducible invariant curves of system (\ref{FNparecidoTakens}).
Hence, any invariant curve of system (\ref{FNparecidoTakens}) can be written as $\CurvI_1^{n_1}\CurvI_2^{n_2}\serU=0$, for some unity formal series $\serU$, and $n_1,n_2\in\N\cup\llave{0}$. 
 
Consequently, the inverse integrating factor can be written as $\facV=\CurvI_1^{n_1}\CurvI_2^{n_2}\serU$, since $\facV=0$ is an invariant curve. 
On the other hand, as $\facV=\rrr +\cdots$, we obtain $n_1+n_2=1$.

As the principal part of $\facV$ is an inverse integrating factor of $\FP_{\lo}$, and those are of the form $\rrr (\rrr -\fracp{2a_{\lo}}{\lo+1}\zzz ^{\lo+1})^{n_2}$ (see also \cite{AGGintnilpfii17,Algaba18}), we deduce $n_2=0$ and then $\facV=\rrr \serU$. 
The condition of inverse integrating factor on $\facV$ leads to:
\bean 
0&=&\nablatr \facV\cdot\FP-\facV \, \diverg\parent{\FP}=\rrr \nablatr \serU\cdot\FP+\serU\nablatr \rrr \cdot\FP-\rrr \serU (G_2'(\zzz ) + 2 G_1(\zzz ))\\
&=&\rrr \fracp{\partial\serU}{\partial \zzz } (\rrr + G_2(\zzz )) + 2 \rrr ^2 \fracp{\partial\serU}{\partial \rrr } G_1(\zzz )+2\serU G_1(\zzz )\rrr -\rrr \serU(G_2'(\zzz )+2 G_1(\zzz ))\\
&=&\rrr \parent{\fracp{\partial\serU}{\partial \zzz } (\rrr + G_2(\zzz ))+2\rrr \fracp{\partial\serU}{\partial \rrr } G_1(\zzz )-\serU G_2'(\zzz )}\\
&=&\rrr\serU^2 \parent{\fracp{2\rrr G_1(\zzz )}{\serU}-\fracp{\partial}{\partial \zzz }\parent{\fracp{\rrr +G_2(\zzz )}{\serU}}}.
\eean
Nevertheless, this equation is incompatible, because the lowest degree quasi-homogeneous term of $\fracp{\partial}{\partial \zzz }\parent{\fracp{\rrr + G_2(\zzz )}{\serU}}$ is $(\mo+1)b_{\mo}\zzz ^{\mo}$, that does not depend on $\rrr $ which is contradictory.

\item $\mo<\lo$. 
The transformation $ \ZZ=\zzz$, $\RR=\rrr +G_2(\zzz )$ leads system (\ref{FNparecidoTakens}) into 
\bea
{\dot{\ZZ }} & = & \RR,\label{systemPtilde}\\
{\dot{\RR }} & = & {(2 G_1(\ZZ ) + G_2'(\ZZ ))\RR- 2 G_1(\ZZ ) G_2(\ZZ )}.\nonumber
\eea
Let us denote the vector field of this system by $\FPt$. 
If $\facV=\rrr +\cdots$ is an inverse integrating factor of system (\ref{FNparecidoTakens}), then system (\ref{systemPtilde}) admits an inverse integrating factor of the form $\facVt=\RR+\cdots$. 
 
The rest of the proof is similar to the previous item.
Taking the type $\tpl=(1,\mo+1)$, the principal part of $\FPt$ is given by
$$ 
\FPt_{\mo}(\ZZ, \RR)=(\RR,(\mo+1)b_{\mo}\ZZ ^{\mo}\RR)^T \in \QH_{\mo}^{\tpl}.
$$
Let us denote $\CurvIt_1=\RR$ and $\CurvIt_2=\RR-b_{\mo}\ZZ ^{\mo+1}+\cdots$.
Then, $\CurvIt_1=0$ and $\CurvIt_2=0$ are the unique irreducible invariant curves of system (\ref{systemPtilde}) (see \cite{AGGintnilp16,AGGintnilpfii17,Algaba18}). 
Hence, there exists some unity formal series $\serUt$ such that the inverse integrating factor of system (\ref{systemPtilde}) is $\facVt=\CurvIt_1^{n_1}\CurvIt_2^{n_2}\serUt$. 

The principal part of $\facVt$ is one of the inverse integrating factors of $\FPt_{\mo}$, which are of the form $\RR(\RR-b_{\mo}\ZZ ^{\mo+1})^{n_2}$. 
Then, we have $n_2=0$ and $\facVt=\RR\serUt$. 
The condition of inverse integrating factor on $\facVt$ leads to:
\bean 
0&=&\nablatr \facVt\cdot\FPt-\facVt\diverg\parent{\FPt}=\RR\nablatr \serUt\cdot\FPt+\serUt\nablatr \RR\cdot\FPt-\RR\serUt(G_2'(\ZZ )+2 G_1(\ZZ ))\\
&&=\RR^2\fracp{\partial\serUt}{\partial \ZZ } +\RR\fracp{\partial\serU}{\partial \RR } ((2 G_1(\ZZ ) + G_2'(\ZZ ))\RR - 2 G_1(\ZZ ) G_2(\ZZ )) - 2 \serUt G_1(\ZZ )G_2(\ZZ ))\\
&=&\RR\parent{\RR\fracp{\partial\serUt}{\partial \ZZ} +\RR \fracp{\partial\serUt}{\partial \RR} (2 G_1(\ZZ ) + G_2'(\ZZ )) - 2 G_1(\ZZ ) G_2(\ZZ )} - 2 \serUt G_1(\ZZ )G_2(\ZZ ).
\eean
Nevertheless, this equation is incompatible, because the lowest degree quasi-homogeneous term of $2 \serUt G_1(\ZZ ) G_2(\ZZ )$ is $2a_{\lo}b_{\mo}\ZZ^{\mo+\lo+1}$, that does not depend on $\RR $ which is contradictory.

\item $\mo=\lo$. 
Taking the type $\tpl=(1,\mo+1)$, the principal part of planar system (\ref{FNparecidoTakens}) is 
$$
\FP_{\mo}(\zzz ,\rrr )=(\rrr +b_{\mo}\zzz ^{\mo+1},2a_{\mo} \zzz ^{\mo}\rrr )^T\in\QH_{\mo}^{\tpl}.
$$
The unique irreducible invariant curves of system (\ref{FNparecidoTakens}) are $\CurvI_1:=\rrr=0$, $\CurvI_2:=\rrr -\fracp{2a_{\mo}-(\mo+1)b_{\mo}}{2(\mo+1)}\zzz ^{\mo+1}+\cdots=0$ (see \cite{AGGintnilp16,AGGintnilpfii17,Algaba18}).
Reasoning as above we obtain a contradiction.
\end{itemize}
\end{proof}

Next result characterizes the existence of analytic first integrals in terms of the existence of  inverse Jacobi multipliers.
\begin{theorem}\label{teHzCnsintMJ} 
System (\ref{hopf-zero-system}) admits an analytic first integral if, and only if, it satisfies some of the following conditions:
\begin{description}
\item[(a)] 
System (\ref{hopf-zero-system}) admits an inverse Jacobi multiplier of the form $\M = \hpp^2 + \cdots$ and it is orbital equivalent to $\F_0+\F_{\mo}+\cdots$ with $\F_{\mo} = \vedor{ a_{\mo}z^{\mo} \FZERO }{ b_{\mo} z^{\mo+1} } $, $\mo\in\N$, $a_{\mo},b_{\mo}\in\R\setminus\llave{0}$, such that $2n_1a_{\mo}+(\mo+1)n_2b_{\mo}=0$ with $n_1,n_2\in\N$ coprimes.
\item[(b)] 
System (\ref{hopf-zero-system}) admits an inverse Jacobi multiplier of the form $\M = \hpp + \cdots$ and it is orbital equivalent to $\F_0$ or $\F_0+\F_{s}+\cdots$ with $\F_{s}$ given in item {\bf (b.1)} or {\bf (b.2)} of Proposition \ref{proCNintHZ}.
\end{description}
\end{theorem}

\begin{proof}
Let us prove the necessary condition. 
Let us assume that system (\ref{hopf-zero-system}) admits an analytical first integral. 
From Theorem \ref{teHzinteg}, it is orbitally equivalent to one of the following normal forms:
\begin{description}
\item[(a)] 
$\dot{\x} = \G(\x):=\F_0(\x) + \vedor{ G_1(\zzz )\FZERO }{ 0 }$, with $G_1(\zzz )=\sum_{k \geq 1} a_k \zzz ^k$. 
In this case, $\Mt = \hpp$ is an inverse Jacobi multiplier for $\G$, and there exists an inverse Jacobi multiplier for $\F$ of the form $\M=\hpp+\cdots$.
\item[(b)] 
$\dot{\x} = \G(\x):=\F_0(\x) + \vedor{ \zero }{ G_2(\zzz ) }$, with $G_2(\zzz )=\sum_{k \geq 1} b_k \zzz ^{k+1}$. 
In this case, $\Mt = \hpp + G_2(\zzz )$ is an inverse Jacobi multiplier for $\G$, and there is an inverse Jacobi multiplier of $\F$ of the form $\M = \hpp + \cdots$.
\item[(c)] 
$\dot{\x} =\G(\x):= \F_0(\x) + \vedor{ a_{\mo}\zzz ^{\mo} \FZERO }{ b_{\mo}\zzz ^{\mo+1} } $. 
In this case, $\diverg\parent{\G}=0$ and $\Mt = \hpp^2 + 2 b_{\mo}\zzz ^{\mo+1} \hpp$ is an inverse Jacobi multiplier for $\G$. 
Hence, there exists an inverse Jacobi multiplier for $\F$ of the form $\M = \hpp^2 + \cdots$.
\item[(d)] 
$\dot{\x} = \G(\x):=\F_0 (\x) + \vedor{ a_{\mo}\zzz ^{\mo} \FZERO}{ b_{\mo} \zzz ^{\mo+1} } $. 
In this case, $\diverg\parent{\G}=(2a_{\mo}+(\mo+1)b_{\mo})\zzz ^{\mo}$ and $\Mt = \hpp^2 - \fracp{2a_{\mo}-(\mo+1)b_{\mo}}{\mo+1}\zzz ^{\mo+1} \hpp$ is an inverse Jacobi multiplier for $\G$. 
Hence, $\F$ admits an inverse Jacobi multiplier of form $\M = \hpp^2 + \cdots$.
\end{description}

Next, we prove the sufficient condition.
We consider the two situations:
\begin{description}
\item[(i)]
If system (\ref{hopf-zero-system}) admits an inverse Jacobi multiplier of form $\M = \hpp^2+\cdots$, the same is true for system (\ref{FNHZEquiv}). 
Moreover, as this system is invariant under rotations, the inverse Jacobi multiplier also is. 
Hence, the planar system (\ref{FNparecidoTakens}) admits an inverse integrating factor of the form $\facV=\rrr ^2+\cdots$. 

If $a_{\mo}b_{\mo}\neq 0$, then the principal part with respect to the type $\tpl=(1,\mo+1)$ of the vector field $\FP$ defined in
(\ref{Pe}) is:
$$\FP_{\mo}(\zzz ,\rrr )=(\rrr +b_{\mo}\zzz ^{\mo+1},2a_{\mo} \zzz ^{\mo}\rrr )^T \in\QH_{\mo}^{\tpl}.$$
Two sub-cases can arise:
\begin{description}

\item[(i.1)]
$n_1 = n_2 = 1$. 
Under this assumption, we have $2a_{\mo}+(\mo+1)b_{\mo}=0$ and $\FP_{\mo}$ is a Hamiltonian vector field, with Hamiltonian function ${-b_{\mo}\zzz ^{\mo+1}\rrr -\fracp{1}{2}\rrr ^2}$. 
The change of variables 
$$
\ZZ=\left((\mo+1)b_{\mo}^2\right)^{1/(2\mo)}\zzz, \ \RR=\left((\mo+1)b_{\mo}^2\right)^{1/(2\mo)}(\rrr +b_{\mo}\zzz ^{\mo+1}), 
$$ 
transforms $\FP$ into a vector field $\FPt$ whose principal part with respect to the type $\tpl=(1,\mo+1)$ is:
$$
\FPt_{\mo}(\ZZ,\RR)=(\RR,\ZZ^{2\mo+1})^T \in\QH_{\mo}^{\tpl}.
$$
As $\FPt$ admits an inverse integrating factor of the form $\RR^2-\fracp{1}{\mo+1}\ZZ^{2(\mo+1)}+\cdots$, from Theorem 6 of \cite{Algaba12} we obtain that $\FPt$ is integrable. 
Using Proposition \ref{proHZTB}, we conclude that system (\ref{FNHZEquiv}) admits a formal first integral and then system (\ref{hopf-zero-system}) also admits a formal first integral. 

\item[(i.2)] 
$n_1\neq n_2$.
Under this assumption, we have $n_12a_{\mo}+n_2(\mo+1)b_{\mo}=0$ and a first integral for $\FP_{\mo}$ is $\rrr ^{n_1}\left(\rrr - c \zzz ^{\mo+1}\right)^{n_2}$, where $c=\fracp{2a_{\mo}-(\mo+1)b_{\mo}}{\mo+1}$. 
The change of variables 
$$
\ZZ=\left(\fracp{(\mo+1)c^2}{4}\right)^{1/(2\mo)}\zzz ,\ \RR=\left(\fracp{(\mo+1)c^2}{4}\right)^{1/(2\mo)}\left(\rrr -\fracp{c}{2}\zzz ^{\mo+1}\right),
$$ 
transforms $\FP$ into a vector field $\FPt$ whose principal part with respect to the type $\tpl=(1,\mo+1)$ is:
$$
\FPt_{\mo}(\ZZ,\RR)=(\RR,\ZZ^{2\mo+1})^T+d\ZZ(\ZZ,(\mo+1)\RR)^T \in\QH_{\mo}^{\tpl},
$$ 
where $d=\fracp{\sqrt{\mo+1}(2a_{\mo}+(\mo+1)b_{\mo})}{|c|}$.
As $\FPt$ admits an inverse integrating factor of the form $\RR^2-\fracp{1}{\mo+1}\ZZ^{2(\mo+1)}+\cdots$, using Theorem 1.3 of \cite{AGGintnilpfii17} we obtain that $\FPt$ is integrable. 
Using Proposition \ref{proHZTB}, we conclude that system (\ref{FNHZEquiv}) admits a formal first integral and then system (\ref{hopf-zero-system}) also admits a formal first integral. 
\end{description}

\item[(ii)]
If system (\ref{hopf-zero-system}) admits an inverse Jacobi multiplier of the form $\M = \hpp + \cdots$, the same is true for system (\ref{FNHZEquiv}). 
Moreover, as this system is invariant under rotations, the inverse Jacobi multiplier also is invariant under rotations. 
Hence, the planar system (\ref{FNparecidoTakens}) admits an inverse integrating factor of the form $\facV=\rrr +\cdots$. 
Let us consider the following three situations:
\begin{description}

\item[(ii.1)] 
$\lo=\mo=\infty$.
This assumption means that system (\ref{hopf-zero-system}) is orbitally equivalent to $\F_0$, and then it admits an analytic first integral.

\item[(ii.2)]  $\lo<\mo$.
Under this assumption, system (\ref{hopf-zero-system}) is orbitally equivalent to $\F_0+\F_{\lo}+\cdots$, where $\F_{\lo} = \vedor{ a_{\lo}\zzz ^{\lo}\FZERO }{ 0 }$, $\lo\in\N$, $a_{\lo}\in\R\setminus\llave{0}$. 
To complete the proof in this case, it is enough to show that $G_2(\zzz )\equiv 0$ in  system (\ref{FNHZEquiv}), because then it admits a formal first integral and then system (\ref{hopf-zero-system}) also admits a formal first integral. 

To show that $G_2(\zzz )\equiv 0$, we use\emph{reductio ad absurdum}.
Let us suppose on the contrary that $G_2(\zzz )\not\equiv 0$.
From Lemma \ref{leNofiir} we obtain that the vector field $\FP$ defined in (\ref{Pe}) does not admit any inverse integrating factor of the form $\rrr +\cdots$, which is contradictory.

\item[(ii.3)]  $\lo>\mo$.
Under this assumption, system (\ref{hopf-zero-system}) is orbitally equivalent to $\F_0+\F_{\mo}+\cdots$, 
where $\F_{\mo} = \vedor{ \zero }{ b_{\mo}\zzz ^{\mo+1} }$, $b_{\mo}\in\R\setminus\llave{0}$. 
To complete the proof in this case, it is enough to show that $G_1(\zzz )\equiv 0$ in system (\ref{FNHZEquiv}), because then it admits a formal first integral and then system (\ref{hopf-zero-system}) also admits a formal first integral. 

To show that $G_1(\zzz )\equiv 0$, we use\emph{reductio ad absurdum}.
Let us suppose on the contrary that $G_1(\zzz )\not\equiv 0$.
From Lemma \ref{leNofiir} we obtain that the vector field $\FP$ defined in (\ref{Pe}) does not admit any inverse integrating factor of the form $\rrr +\cdots$, which is contradictory.
\end{description}
\end{description}
\end{proof}

Before solving specific problems, we present a result that provides an algorithm that characterizes the existence of inverse Jacobi multipliers for system (\ref{hopf-zero-system}).

\begin{lemma}\label{algoritmoJM} 
Let us denote by $\F$ the vector field associated with system (\ref{hopf-zero-system}). 
Then:
\begin{description}
\item[(a)]
There exists a unique scalar function $\M = \hpp^2 + \sum_{k\geq 3}\M_k$, where $\M_k\in\qh_k^\t$, such that the term $\hpp^k$ is missing in $\M_{2k}$ for all $k$, such that 
\beq\label{ecualgJM1} 
\nablatr \M\cdot\F-\M\,\diverg\parent{\F}=\sum_{k\geq 3}\beta_k z^k,
\eeq
where $\beta_k\in\R$.
 
\item[(b)]
There exists a unique scalar function $\Mt = \hpp + \sum_{k\geq 2}\Mt_k$, where $\Mt_k\in\qh_k^\t$, such that the term $\hpp^k$ is missing in $\Mt_{2k}$ for all $k$, such that 
\beq\label{ecualgJM2} 
\nablatr \Mt\cdot\F-\Mt\,\diverg\parent{\F}=\sum_{k\geq 2}\betat_k z^k,
\eeq
where $\betat_k\in\R$.
\end{description}
\end{lemma}

\begin{proof} 
Let us prove item \textbf{(a)}. 
The $k$-degree quasi-homogeneous term of the left-hand side of equality (\ref{ecualgJM1}) is 
$$
\parent{\nablatr \M\,\cdot\F-\M\,\diverg{\F}}_k=\nablatr \M_k\cdot\F_0+  \Rr_k=\Oplt_{k}\parent{\M_k} +   \Rr_k,
$$
where
$$ 
\Rr_k=\sum_{j=s}^{k-1}\parent{\nablatr \M_j\,\cdot\F_{k-j}-\M_j\, \diverg\parent{\F_{k-j}}}\in\qh_k^\t.
$$
To complete the proof it is enough to argue as in the proof of Lemma \ref{algoritmoIPI}. 
The proof of item \textbf{(b)} is analogous.
\end{proof}

\begin{theorem}\label{hhh} 
Let us consider system (\ref{hopf-zero-system}).
\begin{description} 
\item{(a)}
Let assume that the formal orbital normal form (\ref{FNHZEquiv}) falls into the case \textbf{(b.3)} of Proposition \ref{proCNintHZ}, and consider the unique scalar function $\M$ introduced in Lemma \ref{algoritmoJM} satisfying (\ref{ecualgJM1}).
Then, system (\ref{hopf-zero-system}) admits an analytical first integral if, and only if, $\beta_k=0$ for all $k$.

\item{(b)}
Let  assume that the formal orbital normal form (\ref{FNHZEquiv}) falls into the cases \textbf{(a)}, \textbf{(b.1)} or \textbf{(b.2)} of Proposition \ref{proCNintHZ}, and consider the unique scalar function $\Mt$ introduced in Lemma \ref{algoritmoJM} satisfying (\ref{ecualgJM2}).
Then, system (\ref{hopf-zero-system}) admits an analytical first integral if, and only if, $\betat_k=0$ for all $k$.
\end{description} 
\end{theorem}

\begin{proof}
\begin{description} 
\item{(a)}
To prove the necessary condition, we observe that if $\beta_k=0$ for all $k\geq 3$, then $\M$ is an inverse Jacobi multiplier. 
From Proposition \ref{teHzCnsintMJ}, system (\ref{hopf-zero-system}) admits a formal first integral and finally, from Proposition \ref{intanaformal} we deduce that system (\ref{hopf-zero-system}) admits an analytical first integral.

Let us prove the sufficient condition by \emph{reductio ad absurdum}.
Let us suppose on the contrary that there exists a scalar function $\M = \hpp^2 + \sum_{k\geq 3}\M_k$, with $\M_k\in\qh_k^\t$, such that 
$\nablatr \M\cdot\F-\M\,\diverg\parent{\F}=\beta_N z^N+\cdots$, with $\beta_N\neq 0$.
Also, we assume that system (\ref{hopf-zero-system}) admits an analytical first integral.
From Proposition \ref{proCNintHZ}, there exist a time-reparametrization $\mut$ (satisfying $\mut(\zero)=1$) and a near-identity transformation $\x=\phi(\tx)$, bringing system (\ref{hopf-zero-system}) into its orbital normal form (\ref{FNHZEquiv}), which corresponds to the vector field $\G=\F_s+\cdots$, being $\F_s = \vedor{ a_{\mo} z ^{\mo} \FZERO }{ b_{\mo} z ^{\mo+1} }$, where $s=2\mo$ and $2n_1a_{\mo}+(\mo+1)n_2b_{\mo}=0$ for some $n_1,n_2\in\N$ coprime.

Let us define 
$$
\Mh(\tx)=\frac{\mu(\tx)\M\parent{\phi(\tx)}}{\det\parent{D\phi(\tx)}} = \hpp^2 + \cdots.
$$
We have
\beq \label{otraetiq}
\nablatr \Mh\cdot\G-\Mh\,\diverg\parent{\G}=\frac{\det\parent{D\phi(\tx)}}{\mu^2(\tx)} \parent{\nablatr \Mh\cdot\F-\Mh\,\diverg\parent{\F}}=\beta_N z^N+\cdots.
\eeq
On the other hand, it is easy to show that 
$$
\nablatr \Mh\cdot\G-\Mh\,\diverg\parent{\G}=
z^{\mo}\parent{a_{\mo} \nablapl \Mh \cdot \FZERO + b_{\mo} z \frac{\partial \Mh}{\partial z} 
 - \parent{2 a_{\mo}+(\mo+1) b_{\mo}} b_{\mo}\Mh }.
$$
Observe that, as $\beta_N\neq 0$, there is a term of the form $C z^{N-{\mo}}$ in the analytical expression of $\Mh$.
Then, the term $C(N-\mo) z^{N-\mo}$ (which has quasi-homogeneous degree $2N-2\mo<2N$) appears in the analytical expression of $\nablatr\Mh\cdot \F_0$ because it can not be annihilated by the operator $\Oplp_{2N-2\mo}$ (it is not in the range of this linear operator).
But this is contradictory because the left hand side of (\ref{otraetiq}) has quasi-homogeneous degree $2N-2\mo<2N$ and the right hand side of has degree greater than $2N$.

\item{(b)}
The proof of this item is analogous to the proof of Theorem \ref{ggg}. 
\end{description} 
\end{proof}

This theorem allows to define an algorithm for obtaining necessary conditions for the integrability of a polynomial vector field.
Namely, if its formal orbital normal form (\ref{FNHZEquiv}) falls into the case \textbf{(b.3)} of Proposition \ref{proCNintHZ}, it is enough to look for the unique function of the form $\M = \hpp^2 + \cdots$ (specified in item \textbf{(a)} of Lemma \ref{algoritmoJM}) and then discard cases of non-integrability from the conditions $\beta_2\neq 0,\dots$

If the orbital normal form (\ref{FNHZEquiv}) falls into the cases \textbf{(a)}, \textbf{(b.1)} or \textbf{(b.2)} of Proposition \ref{proCNintHZ}, it is enough to look for the unique inverse Jacobi multiplier of the form $\Mt = \hpp + \cdots$ (specified in item \textbf{(b)} of Lemma \ref{algoritmoJM}) and then discard cases of non-integrability from the conditions $\betat_2\neq 0, \dots$ 
In this case, we could also apply the ideas presented in the previous subsection to determine the non-integrability cases.

\section{Some particular cases}
\label{ParticCases}

In this last section, we consider two three-parameter families of vector fields.
The first one corresponds to the family:
\beq\label{sistF0F1v1} 
\vetre{\dot{x}}{\dot{y}}{\dot{z}} = \vetre{-2y}{2x}{x^2+y^2} + \vetre{a_{001}z}{b_{200}x^2}{c_{030}y^3}.
\eeq

Next theorem determines the cases where the above family admits an analytical first integral.

\begin{theorem}\label{teoaplicacion1} 
System (\ref{sistF0F1v1}) admits an analytical first integral if, and only if, $a_{001}=0$.
\end{theorem}

\begin{proof}
A simple computation shows that the first coefficients of the formal orbital normal form (\ref{FNHZEquiv}) for system (\ref{sistF0F1v1}) are:
$$
 a_1 = - 3 \, a_{001}^2 / 8, \ b_1 = 3 \, a_{001}^2 / 8 .
 $$
 Then, two situations can arise:
\begin{description}
\item[(a)] 
If $a_{001}=0$, then system (\ref{sistF0F1v1}) admits the first integral $\IPI(x,y)=x^2+y^2+b_{200}x^3$. 
Moreover, $\M=\IPI$ is an inverse Jacobi multiplier of the form $\M = \hpp + \cdots$.
\item[(b)] 
If $a_{001}\neq 0$, from Theorem \ref{NFHopf-Zero-}, we obtain that $\F$ is orbitally equivalent to the vector field $\F_0+\F_{1}+\cdots$, where $\F_1 = \vedor{ a_{1}z \FZERO }{ b_{1} z^{2} }$ with $a_1$, $b_1$ given before.

As $a_1\neq 0$, $b_1\neq 0$ and $a_1+b_1=0$, from Theorem \ref{teHzCnsintMJ} we obtain that system (\ref{FNHZEquiv}) is integrable if, and only if, it admits an inverse integrating factor of the form $\M = \hpp^2 + \cdots$. 
From Lemma \ref{algoritmoJM}, this occurs if $\alpha_i=0$ for all $2i>4$. 

In this case, we have obtained $\alpha_{6}=\cdots=\alpha_{12}=0$, and 
$$
\alpha_{14}=\fracp{8}{3}a_{001}^8 \parent{126 \, a_{001}^2-117 \, b_{200}a_{001}+40 \, b_{200}^2}.
$$
The vanishing of $\alpha_{14}$ implies $126 \, a_{001}^2-117 \, b_{200}a_{001}+40 \, b_{200}^2 = 0$.
Under this hypothesis, we have obtained $\alpha_{16}=\cdots=\alpha_{20}=0$, and 
\bean 
\alpha_{22}&=&a_{001}^3 \left( 2560 ( 1814374881 \, b_{200} - 1620931148 \, a_{001} ) \, c_{030}^2 \right.\\
&&\left.\hspace{1.3cm} + ( 10742498602234 \, a_{001} - 6174238921423 \, b_{200} ) \, a_{001}^2 \right).
\eean
Assuming now that $ \alpha_{22}$ is zero, we have also obtained that $\alpha_{24}=\alpha_{26}=\alpha_{28}=0$, but $\alpha_{30}\neq 0$.
Hence, system (\ref{sistF0F1v1}) is not integrable in this case.
\end{description}
\end{proof}

The second family that we consider is the following:
\beq\label{sistF0F1v2} 
\vetre{\dot{x}}{\dot{y}}{\dot{z}} = \vetre{-2y}{2x}{x^2+y^2}+\vetre{a_{001}z}{0}{c_{101}xz+c_{011}yz}.
\eeq

Next theorem solves the integrability problem for the above family.

\begin{theorem}\label{teoaplicacion2} 
Let us assume that system (\ref{sistF0F1v2}) admits an analytical first integral. Then, one of the following conditions holds:
\begin{description}
\item[(i)] 
$a_{001}=0$.
\item[(ii)] 
$a_{001}\neq 0$, $a_{001}+c_{011}=0$, $c_{101}=0$.
\item[(iii)] 
$a_{001}\neq 0$, $a_{001}+2c_{011}=0$.
\end{description}
\end{theorem}

\begin{proof}
A simple computation shows that the first coefficients of the normal form (\ref{FNHZEquiv}) for system (\ref{sistF0F1v2}) are:
$$ 
a_1 = - 3 a_{001}(a_{001}+c_{011})/8 , \ b_1= 3 a_{001}(a_{001}+2c_{011})/8 .
$$

We consider the following situations:
\begin{description}
\item[(a)] 
If $a_{001}=0$, then system (\ref{sistF0F1v2}) admits the first integral $\IPI(x,y)=x^2+y^2$. 
Moreover, $\M=\IPI$ is an inverse Jacobi multiplier of the form $\M = \hpp + \cdots$. 
This situation corresponds to item \textbf{(i)} of the statement.

\item[(b)] 
If $a_{001}\neq 0$, from Theorem \ref{NFHopf-Zero-} we obtain that the vector field $\F$ is orbitally equivalente to $\F_0+\F_{1}+\cdots$, where $\F_1 = \vedor{ a_{1}z \FZERO }{ b_{1} z^{2} } $ and $a_1$, $b_1$ are given before.
The following sub-cases can arise:
\begin{description}
\item[(b1)] 
If $a_{001} \neq 0$, $a_{001} + c_{011} = 0$, then $a_1 = 0$, $b_1\neq 0$ and $a_1^2 + b_1^2 \neq 0$. 
From Theorem \ref{teHzCnsintMJ}, if system (\ref{sistF0F1v2}) admits some first integral then there exists an inverse Jacobi multiplier of the form $\M = \hpp + \cdots$. 
From Lemma \ref{algoritmoJM}, $\alpha_i $ must vanish for all $2i>2$. 

In this case, we have obtained $\alpha_4 = \alpha_6 = 0$ and $\alpha_{8} = a_{001}^5 c_{101}/32$, whose vanishing implies that $c_{101}=0$.
With this hypothesis, we have obtained $\alpha_{2j}=0$ for $j=2,\dots, 20$. 
This situation corresponds to item \textbf{(ii)} of the statement.

\item[(b2)] 
If $a_{001} \neq 0$,$a_{001} + 2 c_{011} = 0$, then $a_1 \neq 0$, $b_1= 0$ and $a_1^2 + b_1^2 \neq 0$. 
From Theorem \ref{teHzCnsintMJ}, if system (\ref{sistF0F1v2}) admits some first integral then there exists an inverse Jacobi multiplier of the form $\M = \hpp + \cdots$.

Again, from Lemma \ref{algoritmoJM} we obtain that $\alpha_{2j}$ must vanish for all $j\geq 2$.
We have obtained $\alpha_{2j}=0$ for $j=2,\dots, 18$ and we conjecture that system admits some analytic first integral. 
This situation corresponds to item \textbf{(iii)} of the statement.

\item[(b3)] 
If $a_{001} ( a_{001} + c_{011} ) ( a_{001} + 2 c_{011} ) \neq 0$, then $a_1 b_1 \neq 0$ and using Theorem \ref{teHzCnsintMJ}, we obtain that system (\ref{sistF0F1v2}) is integrable if, and only if, it admits some inverse integrating factor of the form $\M = \hpp^2 + \cdots$ and there are $n_1,n_2\in\N$ coprime such that $2 n_1 a_1 + 2 n_2 b_1 = 0$, i.e., $(n_2-n_1) a_{001} + (2n_2-n_1) c_{011} = 0$.

We notice that $a_{001} = q c_{011} \neq 0$ where $q = - \fracp{2n_2-n_1}{n_2-n_1}\in\Q$.
We claim that $q<-2$ or $q>-1$.
Namely, if $n_2 > n_1$ then $q = - 2 - \fracp{n_1}{n_2-n_1} < - 2$, whereas if $n_2 < n_1$ then $q = - 1 + \fracp{n_2}{n_1-n_2} > - 1$. 

Using this condition, from Lemma \ref{algoritmoJM} we obtain that if system (\ref{sistF0F1v2}) is integrable then $\alpha_{2i}$ must vanish for all $i>2$. 

We have obtained $\alpha_{12} = - c_{011}^7 \, c_{101} \, q^4 (q+2) (5q+8) / {1920}$·. 
We note that $c_{011} \neq 0$, $q \neq 0$, $q + 2 \neq 0$, $5 q + 8 \neq 0$ (otherwise, $n_1$, $n_2$ are not coprime).
Then, the vanishing of $\alpha_{12}$ occurs if, and only if, $c_{101}=0$. 

Under this hypothesis, we have obtained $\alpha_{6} =\alpha_{8} =\alpha_{10} =\alpha_{12} = 0$ and
\bean 
\alpha_{14} & = & c_{011}^{10} \, q^4 \, (3q+5) \, (q+2) \, (q+1) \, (84 q^3 + 201 q^2 + 129 q + 10)
/ 49152,\\
\alpha_{16}&=&0,\\
\alpha_{18}&=& c_{011}^{14} \, q^5 \, (q+2) \, (q+1) \, (308448 q^7 + 2217516 q^6 + 6661397 q^5 + 10859256 q^4 
\\
&&
+ 10383887 q^3 + 5761860 q^2 + 1604516 q + 111300) / {113246208}.
\eean
Using that $\alpha_{14}$ and $\alpha_{18}$ can not vanish simultaneously, we deduce that, in this situation, system (\ref{sistF0F1v2}) does not admit any first integral.
\end{description}
\end{description}
\end{proof}
\vspace{0.5truecm}

\noindent\textbf{Acknowledgements.}
{This research was partly supported by the
\emph{Ministerio de Ciencia e Innovaci\'on, fondos FEDER}
(project MTM2017-87915-C2-1-P),
 by the
\emph{Ministerio de Ciencia, Innovaci\'on y Universidades, fondos FEDER}
(project P6C2018-096265-B-I00)
 and by the \emph{{Consejer\'{\i}a de Econom\'{\i}a, Innovaci\'on, Ciencia y
Empleo} de la Junta de Andaluc\'{\i}a} (projects
{P12-FQM-1658}, TIC-130, FQM-276).
}

\end{document}